\documentclass[11pt]{article}
\usepackage{amsthm,amsfonts, amsbsy, amssymb,amsmath,graphicx}
\usepackage{graphics}
\usepackage[english]{babel}
\usepackage{graphicx}
\usepackage{subcaption}
\usepackage{lineno}
\usepackage{enumerate}
\usepackage{tikz}
\usepackage{tkz-graph}
\usepackage{tkz-berge}
\usepackage{hyperref}
\usepackage{color}
\usepackage{tikz}
\usetikzlibrary{arrows, automata}

\hypersetup{colorlinks=true}

\hypersetup{colorlinks=true, linkcolor=blue, citecolor=blue,urlcolor=blue}
\title{Perfect coalition in graphs}
\author{{\small Doost Ali Mojdeh$^{a}$\thanks{Corresponding author}\ , Mohammad Reza Samadzadeh$^{b}$}\\
{\small $^{a,b}$Department of
		Mathematics, Faculty of Mathematical Sciences}\\{\small University of Mazandaran, Babolsar, Iran}\\
	{\small $^a$damojdeh@umz.ac.ir} \\{\small $^b$m.samadzadeh02@umail.umz.ac.ir}}
\date{}
\newtheorem{theorem}{Theorem}[section]
\newtheorem{corollary}[theorem]{Corollary}
\newtheorem{lemma}[theorem]{Lemma}

\newtheorem{observation}[theorem]{Observation}

\newtheorem{p}{Problem}

\theoremstyle{definition}
\newtheorem{definition}[theorem]{Definition}
\theoremstyle{remark}

\newenvironment{unnumbered}[1]{\trivlist
	\item [\hskip \labelsep {\bf #1}]\ignorespaces\it}{\endtrivlist}

\begin{document}
	
	\maketitle
	\begin{abstract}
		
		\noindent A perfect dominating set in a graph $G=(V,E)$ is a subset $S \subseteq V$ such that each vertex in $V \setminus S$ has exactly one neighbor in $S$. A perfect  coalition in  $G$ consists of two disjoint sets of vertices $V_1$ and $V_2$ 
such that i) neither $V_1$ nor $V_2$ is a dominating set, ii) each vertex in $V(G) \setminus V_1$ has at most one 
neighbor in $V_1$ and each vertex in $V(G) \setminus V_2$ has at most one neighbor in $V_2$, and iii) $V_1 \cup V_2$ is 
a perfect dominating set. A perfect coalition partition (abbreviated $prc$-partition) in a graph $G$ is a vertex partition $\pi= \lbrace V_1,V_2,\dots ,V_k \rbrace$ such that for each set $V_i$ of $\pi$, either $V_i$ is a singleton dominating 
set or there exists a set  $V_j \in \pi$  that forms a perfect coalition with $V_i$. In this paper, we  
initiate the  study of  perfect coalition partitions in graphs.  We obtain a bound on the number of perfect coalitions involving each member of a  perfect coalition partition, 
in terms of maximum degree. The perfect coalition of some special graphs are investigated. Graphs   with minimum degree one, triangle-free graphs  and  trees with  large perfect coalition numbers are investigated.         
\end{abstract}

	\vspace{1mm}
	\section{Introduction}
		Let $G=(V,E)$ denote a simple, finite and undirected   graph of order $n$ with vertex set $V=V(G)$ and edge set $E=E(G)$. The {\em open
		neighborhood} of a vertex $v\in V$ is the set $N(v)=\lbrace u : \ uv \in E \rbrace$, and its {\em closed
		neighborhood} is the set $N[v]=N(v) \cup \lbrace v \rbrace$. Each vertex of $N(v)$ is called a {\em neighbor}
	of $v$, and  the cardinality of $N(v)$ is called the {\em degree} of $v$, denoted by deg$(v)$ or deg$_G (v)$. A
	vertex $v$ of degree $1$ is called a {\em pendant vertex} or {\em leaf}, and its neighbor is called a {\em support vertex}. The {\em minimum} and the {\em maximum
		degree} of $G$ is denoted by $\delta (G)$ and $\Delta (G)$, respectively.  For a set $S$ of vertices of $G$, the subgraph induced by $S$ is denoted by  $G[S]$.
	For two  sets $X$ and $Y$ of vertices, let  $[X,Y]$ denote the set of edges between $X$ and $Y$.	
	If every vertex of $X$ is adjacent to every vertex of  $Y$, we say that $[X,Y]$ is {\em full}, while
	if there are no edges between them, we say that $[X,Y]$ is {\em empty}.
	A subset $V_i \subseteq V$ is called a {\em singleton set}, {\em doubleton set} and {\em tripleton set}, if $\lvert V_i \rvert =1$, $\lvert V_i \rvert =2$ and $\lvert V_i \rvert =3$, respectively. We denote the path, cycle and complete graph of order $n$, by $P_n$, $C_n$ and $K_n$, respectively. The girth of a graph $G$, denoted by $g(G)$, is the length of its shortest cycle. A graph is called triangle-free if it has no $K_3$ as a subgraph. 
	
	A set $S\subseteq V$  in a graph $G=(V,E)$ is called a {\em dominating set} if  for any vertex $v\in V$,  either $v \in S$ or $v$ has a neighbor in  $S$. The minimum cardinality of a dominating set of $G$ is called {\em domination number} of $G$, denoted by $\gamma (G)$.  A set $S \subseteq V$ is a  {\em perfect dominating set} of a graph $G = (V ,E)$ if every vertex in $V \setminus S$ has exactly one neighbor in $S$. We denote the minimum cardinality of of a perfect dominating set of $G$, by $\gamma_{p} (G)$. The concept of domination and its variants have been widely studied in the literature, see for example \cite{ref10,ref4,ref11}.

In 2020,  Haynes et al. \cite{ref7} introduced the concept of coalition domination in graphs.
 A {\em coalition} in a graph $G=(V,E)$ consists of two disjoint sets $V_1$ and $V_2$ of vertices such that neither $V_1$ nor $V_2$ is a dominating set, but the union $V_1\cup V_2$
is a dominating set of $G$. A {\em coalition partition} (abbreviated $c$-partition) in a graph $G$  is
a vertex partition $\pi = \lbrace V_1, V_2, \dots , V_k\rbrace$ such that every set $V_i$ either is a singleton dominating set, or is not a dominating set but
forms a coalition with another set $V_j$ in $G$. The maximum cardinality of a coalition partition in $G$
is called the {\em coalition number} of  $G$,  denoted by $C(G)$. A $c$-partition of $G$ with order $C(G)$ is called a $C(G)$-partition. For more studies in this area, we refer the reader to \cite{ref3,ref6,ref9,ref5,ref8}. Since the introduction of this concept, some of its variants have been introduced and studied, see for example \cite{ref1,ref2,ref12,ref13}. In this paper, we introduce the concepts of perfect coalition and perfect coalition partition in graphs. We define these terms as follows.

		\begin{definition}\label{prc-def}
		A {\em perfect coalition} in a graph $G$ consists of two disjoint sets of vertices $V_1$ and $V_2$ such that 
		\begin{enumerate}
	\item[(i)] 
	Neither $V_1$ nor $V_2$ is a dominating set of $G$.
	\item[(ii)]
	Each vertex in $V(G) \setminus V_1$ has at most one neighbor in $V_1$, and each vertex in $V(G) \setminus V_2$ has at most one neighbor in $V_2$.
	\item[(iii)] 
	$V_1 \cup V_2$ is a perfect dominating set of $G$.
	\end{enumerate}
	\end{definition}
	\begin{definition}\label{prcp-def}
A {\em perfect coalition partition} (abbreviated $prc$-partition) in a graph $G$ is a vertex partition $\pi= \lbrace V_1,V_2,\dots ,V_k \rbrace$ such that for each set $V_i$ of $\pi$ either $V_i$ is a singleton dominating set, or there exists a set  $V_j \in \pi$  that forms a perfect coalition with $V_i$. The maximum cardinality of a $prc$-partition in $G$ is called the {\em perfect coalition number} of $G$, denoted by $PRC(G)$. We say that $PRC(G)=0$ if $G$ has no $prc$-partition. A $prc$-partition of $G$ with order $PRC(G)$ is called a $PRC(G)$-partition. 
	\end{definition}

As will be seen below, Section 2 investigates the number of $prc$-partners of a set in a  $prc$-partition of a graph  in terms of maximum degree. In Section 3, we determine  perfect coalition number of paths and cycles. In Section 4, we characterize the graphs $G$ with $\delta(G)=1$ and $PRC(G)=\lvert V(G)\rvert$ in Subsection 4.1, while in Subsection 4.2, we characterize triangle-free graphs $G$ with perfect coalition number $\lvert V(G)\rvert$. The trees $T$ of order $n$ are characterized whenever $PRC(T)\in \{n,n-1,n-2\}$ in Subsection 4.3.
Finally, we close the paper with some research problems in Section 5.
\section{Bounds}
Obviously, every $prc$-partition of a graph $G$ is a $c$-partition of $G$. Thus, we have the following.
\begin{observation} \label{bound-C}
For any graph $G$, $PRC(G) \leq C(G)$.
\end{observation} 
Haynes et al. \cite{ref7} proved the following result.
\begin{theorem}
	\cite{ref7} Let $G$ be a graph with maximum degree $\Delta (G)$, and let $\pi$ be a $C(G)$-partition. If $X \in \pi$, then $X$ is in at most $\Delta (G)+1$ coalitions.
\end{theorem}  
The following theorem shows that for a graph $G$, the above upper bound can be reduced to $\Delta (G)$, when dealing with perfect coalition partitions. 
\begin{theorem} \label{bound-delta}
	Let $G$ be a connected graph with maximum degree $\Delta (G)$, and let $\pi$ be a $prc$-partition of $G$. If $A \in \pi$, then $A$ is in at most $\Delta (G)$ perfect coalitions. Further, this bound is sharp.
\end{theorem}
\begin{proof}
	If $A$ is  a singleton dominating set, then it has no $prc$-partner. Hence, we may assume that $A$ does not dominate $G$. Let $x$ be a vertex that is not dominated by $A$. Now $x$ must be dominated by every $prc$-partner of $A$, implying that every $prc$-partner of $A$ must contain a vertex in $N[x]$. Since $\lvert N[x]\rvert \leq \Delta (G)+1$, it follows that  $A$ cannot have more that $\Delta (G)+1$ $prc$-partners. Now we show that $A$ has at most $\Delta (G)$ $prc$-partners. Suppose, to the contrary, that $A$ admits $\Delta (G)+1$ $prc$-partners. Let $V_1,V_2,\dots ,V_{\Delta (G)+1}$ be $prc$-partners of $A$. Define $G^\prime =G[U]$, where $U=\bigcup_{i=1} ^{\Delta (G)+1} V_i$.  Let $X$ be the set of vertices in $G^\prime$  having a neighbor in $A$, and $Y$ be the set of vertices $v_i$ in $G^\prime$  having a neighbor in $U \setminus V_{v_i}$, where $V_{v_i}$ is the set in $\pi$ containing $v_i$. 
	
	First we show that $X \cap Y= \emptyset$. Suppose, to the contrary, that $X \cap Y \neq \emptyset$. Let $s \in X \cap Y$, where $V_s$ is the set in $\pi$ containing $s$. Further,  let $t \in U \setminus V_s$ be a neighbor of $s$ in $G$, and let $V_t$ be the set in $\pi$ containing $t$. Now we observe that $s$ has at least two neighbors in $A \cup V_t$, implying that the sets $A$ and $V_t$ are not $prc$-partners, a contradiction. Hence, $X \cap Y =\emptyset$. 

Next we show that $X \cup Y =U$. Suppose, to the contrary, that $X \cup Y \neq U$. Let $e \in U \setminus (X \cup Y)$. Let $V_e$ be the set in $\pi$ containing $e$. Now for each $V_i \in \lbrace V_1,V_2,\dots ,V_{\Delta (G)+1} \rbrace \setminus \lbrace V_e\rbrace$, $e$ is not dominated by $A \cup V_i$, implying that $A$ and $V_i$ are not $prc$-partners, a contradiction. Hence, $X \cup Y =U$.
	
	As before, we let $x$ denote a vertex in $G$ that is not dominated by $A$. Note that $N[x] \subseteq  U \setminus X =Y$, and so $Y \neq \emptyset$. Let $f$ be an arbitrary vertex in $Y$, where $V_f$ is the set in $\pi$ containing $f$. Now for each $V_i \in \lbrace V_1,V_2, \dots ,V_{\Delta (G)+1}\rbrace \setminus \lbrace V_f\rbrace$, since $V_i \cup A$ is a dominating set of $G$ and $A$ does not dominate $f$, it follows that $f$ has a neighbor in $V_i$. Therefore, deg$_{G^\prime} (f)\geq \Delta (G)$, and so deg$ (f) =$ deg$_{G^\prime} (f) = \Delta (G)$. Now, choosing $f$ arbitrarily, we deduce that $[Y, V(G) \setminus Y]=\emptyset$ which contradicts the fact that $G$ is connected. Hence, $A$ admits at most $\Delta (G)$ $prc$-partners.
	
	To prove the sharpness, for each $\Delta \geq 2$, we construct a graph  $G_\Delta$ with $\Delta (G_\Delta)=\Delta$, and a $prc$-partition $\pi$  such that there exists a set in $\pi$ having $\Delta$ $prc$-partners. Let $G_\Delta$ be a graph with $V(G_\Delta)= \lbrace w,v, u_1,u_2,\dots ,u_\Delta \rbrace$. and $E(G_\Delta) =\lbrace u_i u_j : 1\leq i <j \leq \Delta \rbrace \cup \lbrace wv, vu_1 \rbrace$. Consider the vertex partition $\pi =\lbrace \lbrace w\rbrace ,\lbrace u_1,v\rbrace ,\lbrace u_2\rbrace ,\lbrace u_3\rbrace ,\dots ,\lbrace u_\Delta \rbrace \rbrace$. One can observe that $\pi$ is a $prc$-partition of $G_\Delta$, where the set $\lbrace w\rbrace$ forms a perfect coalition with all other sets of $\pi$. This completes the proof.
	
\end{proof}
Note that the bound presented in Theorem \ref{bound-delta} does not hold for disconnected graphs. For example, consider the graph $G=K_m \cup K_2$, $m\geq 2$, with $V(G)= \lbrace v_1,v_2,\dots ,v_m \rbrace \cup \lbrace u_1,u_2\rbrace $ , and $E(G)= \lbrace v_i v_j \  | \ 1\leq i\leq j \leq m \rbrace \cup \lbrace u_1 u_2\rbrace$.  One can observe that the singleton partition $\pi_1$ of $G$ is a $prc$-partition for $G$, where the set $\lbrace u_1\rbrace $ has $m$ $prc$-partners, while $\Delta (G)=m-1$. Hence we obtain the following result.
\begin{observation} \label{bound-dis}
		Let $G$ be a disconnected graph with maximum degree $\Delta (G)$, and let $\pi$ be a $prc$-partition of $G$. If $A \in \pi$, then $A$ is in at most $\Delta (G)+1$ perfect coalitions. Further, this bound is sharp.
\end{observation}

\section{Paths and cycles}
In this section, we will determine the perfect coalition number of  paths and cycles. 

\begin{lemma} \cite{ref6} \label{lempath}
	For any path $P_n$, $C(P_n)\leq 6$.
\end{lemma}
Using Lemma \ref{lempath} and Observation \ref{bound-C}, we have the following.
\begin{corollary}\label{l-path}
	For any path $P_n$, $PRC(P_n)\leq 6$.
\end{corollary}

\begin{theorem} \label{prc-path}  
	For the path $P_n$, 
	$$
	PRC(P_n)= \begin{cases}
		1 & if \  n=1, \\
		2 & if \  n=2, \\
		0 & if \  n=3, \\
		4 & if \  n=4,6,8, \\
		3 & if \  n=5, \\
		5 & if \  n=7,9,10,11,13, \\
		6 & otherwise.
	\end{cases}
	$$ 
\end{theorem}
\begin{proof}
	Consider the path $P_n$ with $V(P_n)=\lbrace v_1,v_2,\dots ,v_n\rbrace$ and $E(P_n)=\{ v_i v_{i+1} : 1\leq i\leq n-1\}$.
	The result is obvious for $n \leq 4$, so we assume $n\geq 5$. For the paths $P_5$ and $P_6$, it is easy to check that $P_5$ has no $prc$-partition of order $5$ or $4$, and that $P_6$ has no $prc$-partition of order $6$ or $5$. Now the partitions $\lbrace \lbrace v_1,v_5\rbrace ,\lbrace v_2\rbrace ,\lbrace v_3,v_4\rbrace \rbrace$ and $\lbrace \lbrace v_2\rbrace ,\lbrace v_5\rbrace ,\lbrace v_3,v_6\rbrace ,\lbrace v_1,v_4\rbrace \rbrace$ are $prc$-partitions for $P_5$ and $P_6$, respectively. Therefore, $PRC(P_5)=3$ and $PRC(P_6)=4$. 
	
	Now let $n=7$. By Theorem \ref{bound-delta}, $PRC(P_7) \neq 6$. On the other hand, the partition $\lbrace \lbrace v_2,v_6\rbrace, \lbrace v_1,v_7\rbrace , \lbrace v_3\rbrace , \lbrace v_4\rbrace ,\lbrace v_5\rbrace \rbrace$ is a $prc$-partition for $P_7$, and so $PRC(P_7)=5$. 
	
	Next assume $n=8$. First we show that $PRC(P_8) \neq 6$. Suppose the opposite is right. By Theorem \ref{bound-delta}, a $PRC(P_8)$-partition consists of two doubleton sets and four singleton sets, where each singleton set forms a perfect coalition with a doubleton set, and each doubleton set forms a perfect coalition with exactly two singleton sets. This contradicts the fact that the vertices $v_3$ and $v_6$ are not present in any such perfect coalitions. Hence, $PRC(P_8) \neq 6$. Now we show that $PRC(P_8)\neq 5$. Suppose, to the contrary, that $PRC(P_8)= 5$. Let $\pi$ be a $PRC(P_8)$-partition with $5$ sets. By Theorem \ref{bound-delta}, it suffices to examine the following cases.
	 \begin{itemize}
	 	\item  $\pi$ consists of a tripleton set, a doubleton set and three singleton sets. 
	 
	 Since $\gamma_{p} (P_8)=3$, it follows that each singleton set must form a $prc$-coalition with a non-singleton set. Now by Theorem \ref{bound-delta}, there exists two distinct singleton sets in $\pi$, say $V_1$,$V_2$, such that $V_1$ forms a $prc$-coalition with the doubleton set and $V_2$ forms a $prc$-coalition with the tripleton set. On the other hand, $P_8$ has six  perfect dominating sets of order $4$, namely, \\
	  	$N_1 =\lbrace v_1,v_2,v_5,v_8\rbrace$,
	  	$N_2=\lbrace v_1,v_4,v_5,v_8\rbrace$,
	  	$N_3=\lbrace v_1,v_4,v_7,v_8\rbrace$, 
	  	$N_4=\lbrace v_2,v_5,v_6,v_7\rbrace$,
	  	$N_5=\lbrace v_2,v_3,v_6,v_7\rbrace$,
	  	$N_6=\lbrace v_2,v_3,v_4,v_7\rbrace$, 
	and two perfect dominating sets of order $3$, namely,
	$T_1=\lbrace v_2,v_5,v_8\rbrace$ and $T_2=\lbrace v_1,v_4,v_7\rbrace$, where $N_i \cap T_j  \neq \emptyset$, for each $1 \leq i\leq 6$ and $j=1,2$, a contradiction. 
	 
	\item  $\pi$ consists of three doubleton sets and two singleton sets.
	 
	  Considering the fact that $P_8$ has exactly two perfect dominating sets of order $3$, we can easily derive a contradiction. Hence, $PRC(P_8)\neq 5$.

The partition $\lbrace \lbrace v_2,v_3,v_6\rbrace ,\lbrace v_1,v_4,v_5\rbrace ,\lbrace v_7\rbrace ,\lbrace v_8\rbrace \rbrace$ is a $prc$-partition for $P_8$, so $PRC(P_8)=4$.
\end{itemize}
Next assume $n=9$. First we show that $PRC(P_9) \neq 6$. Suppose that the converse is true. By Theorem \ref{bound-delta}, it suffices to examine the following cases.
\begin{itemize}
	
\item  $\pi$ consists of a tripleton set, a doubleton set and four singleton sets.

 Since $\gamma_{p} (P_9) =3$, each singleton set must form a perfect coalition with a non-singleton set. Thus, by Theorem \ref{bound-delta}, each non-singleton set must form a perfect coalition with two singleton sets. But $P_9$ has a unique perfect dominating set of order $3$, namely $\lbrace v_2,v_5,v_8\rbrace$, a contradiction. 

	\item  $\pi$ consists of three doubleton sets and three singleton sets.
	 
	  Similar to the previous case, we can easily derive a contradiction. Hence, $PRC(P_9) \neq 6$.
	 
	 On the other hand, the partition $\lbrace \lbrace v_3,v_6,v_9\rbrace ,\lbrace v_1,v_4,v_8\rbrace ,\lbrace v_2\rbrace ,\lbrace v_5\rbrace ,\lbrace v_7\rbrace \rbrace$ is a $prc$-partition for $P_9$, so $PRC(P_9)=5$.
\end{itemize}
Next assume $n=10$. First we show that $PRC(P_9) \neq 6$. Suppose that the converse is true. Using Theorem \ref{bound-delta} and the fact that $\gamma_{p} (P_{10})=4$, we only need to examine the case in which $\pi$ consists of two sets of order $3$ and four singleton sets. In this case, each singleton set must form a perfect coalition with a set of order $3$, and that each set of order $3$ must form a perfect coalition with exactly two singleton sets. Consequently, there must be four perfect dominating sets (name $A$,$B$,$C$ and $D$) of $P_{10}$ of order $4$ such that  these sets can be partitioned into two pair, say $(A,B)$ and $(C,D)$, with the property that the sets in each pair have exactly three vertices in common. But $P_{10}$ has four perfect dominating sets of order $4$, namely $V_1= \lbrace v_1,v_4,v_7,v_{10}\rbrace$, $V_2= \lbrace v_2,v_3,v_6,v_9\rbrace$, $V_3= \lbrace v_2,v_5,v_6,v_9\rbrace$ and $V_4= \lbrace v_2,v_5,v_8,v_9\rbrace$, where $V_1$ has no common vertex with any other sets, a contradiction.
The partition $\lbrace \lbrace v_1,v_7\rbrace ,\lbrace v_4,v_{10}\rbrace ,\lbrace v_2,v_9\rbrace ,\lbrace v_3,v_6\rbrace ,\lbrace v_5,v_8\rbrace \rbrace$ is a $prc$-partition for $P_{10}$, so  $PRC(P_{10})=5$. 

Next assume $n=11$. Using Theorem \ref{bound-delta}, the fact that $\gamma_{p} (P_{11})=4$ and the fact that $P_{11}$ has exactly two perfect dominating sets of order $4$, namely $\lbrace v_1,v_4,v_7,v_{10}\rbrace$ and $\lbrace v_2,v_5,v_8,v_{11} \rbrace$, it is not hart to verify that $PRC(P_{11})\neq 6$. Now the partition \\
 $\lbrace \lbrace v_1,v_4,v_8,v_{11}\rbrace ,\lbrace v_2,v_3,v_6,v_{10}\rbrace ,\lbrace v_5\rbrace ,\lbrace v_7\rbrace,\lbrace v_9\rbrace \rbrace$ is a $prc$-partition for $P_{11}$, so $PRC(P_{11})=5$.

Next assume $n=13$. The partition $\lbrace  \lbrace v_1,v_4,v_5,v_8,v_{12} \rbrace,\lbrace v_2,v_3,v_7,v_{10}, v_{13}\rbrace ,\lbrace v_6\rbrace ,\lbrace v_9\rbrace ,\lbrace v_{11}\rbrace \rbrace$ is a $prc$-partition for $P_{13}$, so $PRC(P_{13})\geq 5$. Now we show that $PRC(P_{13}) \neq 6$. Suppose that the converse is true. Let $\pi$ be a $PRC(P_{13})$-partition.
Note that $P_{13}$ has five perfect dominating sets of order $5$, namely, 
$S_1= \lbrace v_1,v_4,v_7,v_{10},v_{13} \rbrace $, 
$S_2= \lbrace v_2,v_3,v_6,v_9,v_{12} \rbrace $, 
$S_3= \lbrace v_2,v_5,v_6,v_9,v_{12} \rbrace $, 
$S_4= \lbrace v_2,v_5,v_8,v_9,v_{12} \rbrace $, 
$S_5= \lbrace v_2,v_5,v_8,v_{11},v_{12} \rbrace $, 
and twenty perfect dominating sets of order $6$, namely, 
$M_1 =\lbrace v_1,v_2,v_5,v_6,v_9,v_{12}\rbrace$, 
$M_2 =\lbrace v_2,v_3,v_6,v_7,v_{10},v_{13}\rbrace$, 
$M_3 =\lbrace v_1,v_2,v_5,v_8,v_9,v_{12}\rbrace$, \\ 
$M_4 =\lbrace v_2,v_3,v_6,v_9,v_{10},v_{13}\rbrace$, 
$M_5 =\lbrace v_1,v_2,v_5,v_8,v_{11},v_{12}\rbrace$, 
$M_6 =\lbrace v_2,v_3,v_6,v_9,v_{12},v_{13}\rbrace$, \\ 
$M_7 =\lbrace v_1,v_4,v_5,v_8,v_9,v_{12}\rbrace$, 
$M_8 =\lbrace v_2,v_3,v_6,v_9,v_{12},v_{13}\rbrace$, 
$M_9 =\lbrace v_1,v_4,v_5,v_8,v_{11},v_{12}\rbrace$,  \\
$M_{10} =\lbrace v_2,v_3,v_6,v_9,v_{12},v_{13}\rbrace$, 
$M_{11} =\lbrace v_1,v_4,v_7,v_8,v_{11},v_{12}\rbrace$, 
$M_{12} =\lbrace v_2,v_5,v_8,v_9,v_{12},v_{13}\rbrace$,  \\
$M_{13} =\lbrace v_1,v_2,v_3,v_6,v_9,v_{12}\rbrace$, 
$M_{14} =\lbrace v_2,v_3,v_4,v_7,v_{10},v_{13}\rbrace$, 
$M_{15} =\lbrace v_1,v_4,v_5,v_6,v_9,v_{12}\rbrace$,  \\
$M_{16} =\lbrace v_2,v_5,v_6,v_7,v_{10},v_{13}\rbrace$, 
$M_{17} =\lbrace v_1,v_4,v_7,v_8,v_9,v_{12}\rbrace$, 
$M_{18} =\lbrace v_2,v_5,v_8,v_9,v_{10},v_{13}\rbrace$,  \\ 
$M_{19} =\lbrace v_1,v_4,v_7,v_{10},v_{11},v_{12}\rbrace$, 
$M_{20} =\lbrace v_2,v_5,v_8,v_{11},v_{12},v_{13}\rbrace$.

Observe that for each $1\leq i\leq 5$ and $1 \leq j\leq 20$, $S_i \cap M_j \neq \emptyset$.
 Now using Theorem \ref{bound-delta}, and the fact that $\gamma_{p} (P_{13})=5$, the only non-trivial cases to examine are as follows.

\begin{itemize}
\item  $\pi$ consists of a set of order $5$, a set of order $4$ and four singleton sets. 

By Theorem \ref{bound-delta}, each singleton set must form a $prc$-coalition with exactly one non-singleton set, and each non-singleton set must form a $prc$-coalition with exactly two singleton sets. Thus, $P_{13}$ must contain two perfect dominating sets $A$ and $B$, such that $\lvert A\rvert =5$, $\lvert B\rvert =6$ and $A \cap B =\emptyset$, a contradiction.

\item  $\pi$ consists of two sets of order $4$, a doubleton set and four singleton sets. 

An argument similar to the one used in previous case indicates that this case is also impossible.

\item  $\pi$ consists of a set of order $4$, a tripleton set, two doubleton sets and two singleton sets. 

Considering Theorem \ref{bound-delta}, and the fact that $\gamma_{p} (P_{13})=5$, we deduce that each singleton set must form a $prc$-coalition with the set of order $4$, and each doubleton set must form a $prc$-coalition with the set of order $3$. This implies that $P_{13}$ contains four perfect dominating sets (name $A$,$B$,$C$ and $D$) of order $5$ such that $A \cap B =\emptyset$, and $C \cap D =\emptyset$ which is a contradiction since for each $2 \leq i \leq j \leq 5$, we have $S_i \cap S_j \neq \emptyset$.  Hence, $PRC(P_{13})\neq 6$, and so $PRC(P_{13})=5$.
\end{itemize}
If $n\geq 12$, and   $n \equiv 0  \pmod{2} $, then the partition$
\lbrace V_1=\lbrace v_2,v_6,v_9\rbrace \cup \lbrace v_{4k} : k\geq 3\rbrace \cup \lbrace v_{4k+1} : k\geq 3\rbrace , V_2=\lbrace v_1,v_4,v_7\rbrace \cup \lbrace v_{4k-1} : k\geq 3\rbrace \cup \lbrace v_{4k+2} : k\geq 3 \rbrace , V_3=\lbrace v_5\rbrace , V_4=\lbrace v_3\rbrace , V_5=\lbrace v_{10}\rbrace ,V_6 =\lbrace v_8\rbrace \rbrace  
$ is a $prc$-partition of order $6$ for $P_n$.

Finally, if  $n\geq 15$, and   $n \equiv 1  \pmod{2} $, then the partition \\
$
\lbrace V_1=\lbrace v_2,v_9,v_{12}\rbrace \cup \lbrace v_{4k-1} : k\geq 4\rbrace \cup \lbrace v_{4k} : k\geq 4\rbrace , V_2=\lbrace v_1,v_4,v_7,v_{10}\rbrace \cup \lbrace v_{4k-2} : k\geq 4\rbrace \cup \lbrace v_{4k+1} : k\geq 4 \rbrace , V_3=\lbrace v_5,v_8\rbrace , V_4=\lbrace v_3,v_6\rbrace , V_5=\lbrace V_{13}\rbrace ,V_6 =\lbrace v_{11}\rbrace \rbrace  
$ is a $prc$-partition of order $6$ for $P_n$. Thus the proof is observed.
\end{proof}

\begin{lemma} \emph{\cite{ref6}} \label{lemcycle}
	For any cycle $C_n$, $C(C_n)\leq 6$.
\end{lemma}
Using Lemma \ref{lemcycle} and Observation \ref{bound-C}, we have the following.
\begin{corollary}\label{l-cycle}
	For any cycle $C_n$, $PRC(C_n)\leq 6$.
\end{corollary}
\begin{lemma} \label{lem-c}
	If $n\geq 9$ and $n \neq 11$, then $PRC(C_n) =6$.
\end{lemma}
\begin{proof}
		Consider the cycle $C_n$ with $V(C_n)=\lbrace v_1,v_2,\dots ,v_n\rbrace$ and $E(C_n)=\lbrace v_i v_{i+1} : 1\leq i\leq n-1\rbrace \cup \lbrace v_1 v_n\rbrace$.
	The partitions \\
	$
	\pi_9= \lbrace \lbrace v_1,v_7\rbrace, \lbrace v_2,v_8\rbrace ,\lbrace v_3 ,v_9\rbrace ,\lbrace v_4\rbrace ,\lbrace v_5\rbrace, \lbrace v_6\rbrace \rbrace,
$ \\
	$
	\pi_{12}= \lbrace \lbrace v_1,v_7\rbrace, \lbrace v_2,v_8\rbrace ,\lbrace v_3 ,v_9\rbrace ,\lbrace v_4,v_{10}\rbrace ,\lbrace v_5,v_{11}\rbrace, \lbrace v_6,v_{12}\rbrace \rbrace 
	$ and \\
	$
	\pi_{15}= \lbrace \lbrace v_1,v_7,v_{13}\rbrace, \lbrace v_2,v_8,v_{14}\rbrace ,\lbrace v_3 ,v_9,v_{15}\rbrace ,\lbrace v_4,v_{10}\rbrace ,\lbrace v_5,v_{11}\rbrace, \lbrace v_6,v_{12}\rbrace \rbrace
	$ \\
	are $prc$-partitions for $C_9$, $C_{12}$ and $C_{15}$, respectively.
	Thus, using Corollary \ref{l-cycle}, we have $PRC(C_9)=PRC(C_{12})=PRC(C_{15})=6$. Now let $\pi =\lbrace V_1,V_2,\dots ,V_6\rbrace$ be a $prc$-partition for $C_n$. We say that $\pi$ is proper, if it has the following properties.
	\begin{enumerate}
		\item[(i)] 
		$V_2$ and $V_3$ are $prc$-partners of $V_1$.
		\item[(ii)] 
		$V_5$ and $V_6$ are $prc$-partners of $V_4$.
		\item[(iii)] 
		$v_1 \in V_1$ and $v_n \in V_4$.
	\end{enumerate}
	To complete the proof, we show that if $n\geq 10$ and $n\neq 11,12$, then $C_n$ has a proper $prc$-partition. We proceed by induction on $n$. Consider the partitions \\
	$\pi_{10} = \lbrace V_1= \lbrace v_1,v_5,v_8\rbrace ,V_2=\lbrace v_4\rbrace, V_3=\lbrace v_2\rbrace , V_4=\lbrace v_3,v_6,v_{10}\rbrace ,V_5=\lbrace v_7\rbrace ,V_6 =\lbrace v_9\rbrace \rbrace, $ \\
	$\pi_{13} = \lbrace V_1= \lbrace v_1,v_5,v_8,v_{11}\rbrace ,V_2=\lbrace v_2\rbrace, V_3=\lbrace v_4\rbrace , V_4=\lbrace v_3,v_6,v_{13}\rbrace ,V_5=\lbrace v_9,v_{12}\rbrace ,V_6 =\lbrace v_7, v_{10}\rbrace \rbrace,$ \\
$\pi_{16} = \lbrace V_1= \lbrace v_1,v_5,v_8,v_{11},v_{14}\rbrace ,V_2=\lbrace v_4\rbrace, V_3=\lbrace v_2,\rbrace , V_4=\lbrace v_3,v_6,v_{16}\rbrace ,V_5=\lbrace v_9,v_{12},v_{15}\rbrace ,V_6 =\lbrace v_7, v_{10},v_{13}\rbrace \rbrace$ and	\\
$\pi_{19} = \lbrace V_1= \lbrace v_1,v_8,v_{11},v_{14}, v_{17}\rbrace ,V_2=\lbrace v_4,v_7\rbrace, V_3=\lbrace v_2,v_5\rbrace , V_4=\lbrace v_3,v_6,v_9,v_{19}\rbrace ,\linebreak V_5=\lbrace v_{12},v_{15},v_{18}\rbrace ,V_6 =\lbrace v_{10},v_{13},v_{16}\rbrace \rbrace $. \\	
These partitions are respectively, proper $prc$-partitions for $C_{10}$, $C_{13}$, $C_{16}$ and $C_{19}$. This establishes the base case. Now assume that $\pi =\lbrace V_1, V_2, \dots ,V_6\rbrace$ is a proper $prc$-partition for $C_k$. Consider the partition $\pi^\prime =\lbrace V_1 \cup \lbrace v_{k+1},v_{k+2}\rbrace  , V_2,V_3,V_4 \cup \lbrace v_{k+3} ,v_{k+4} \rbrace ,V_5,V_6\rbrace $. One can observe that $\pi^\prime$ is a proper $prc$-partition for $C_{k+4}$. This completes the proof.
\end{proof}
\begin{theorem} \label{prc-path}  
	For the cycle $C_n$, 
	$$
	PRC(C_n)= \begin{cases}

		n & if \  n=3,4,6, \\
		3 & if \  n=5, \\
		4 & if \  n=8, \\
		5 & if \  n=7,11, \\
		6 & otherwise.
	\end{cases}
	$$ 
\end{theorem}
\begin{proof}
		Consider the cycle $C_n$ with $V(C_n)=\lbrace v_1,v_2,\dots ,v_n\rbrace$ and $E(C_n)=\lbrace v_i v_{i+1} : 1\leq i\leq n-1\rbrace \cup \lbrace v_1 v_n\rbrace$.
	The result is obvious for $n \in \lbrace 3,4,6\rbrace$. Now let $n=5$. Note that the singleton vertex partition of $C_5$ is not a $prc$-partition for $C_5$, so $PRC(C_5) \neq 5$. Note also that since $C_5$ has no perfect dominating set of order $2$, it follows from Theorem \ref{bound-delta} that $PRC(C_5) \neq 4$. The partition $\lbrace \lbrace v_1\rbrace ,\lbrace v_2,v_3\rbrace ,\lbrace v_4,v_5\rbrace \rbrace$ is a $prc$-partition for $C_5$, so $PRC(C_5)=3$.

Next assume $n=7$. Considering Theorem \ref{bound-delta} and $\gamma_{p} (C_7)=3$, we deduce that $PRC(C_7) \neq 6$. Now we show that $PRC(C_7) = 5$.  Let $\pi =\lbrace V_1,V_2,V_3,V_4,V_5\rbrace$ where $V_1=\{v_1,v_7\}$, $V_2=\{v_2,v_6\}$, 
$V_3=\{v_3\}$, $V_4=\{v_4\}$, $V_5=\{v_5\}$. Then $\pi$ is a $prc$-partition for $C_7$. 

	Next assume $n=8$. Considering Theorem \ref{bound-delta} and $\gamma_{p} (C_8)=4$, it is not difficult to see that $PRC(C_8) \neq 6$, and that $PRC(C_8) \neq 5$. The partition $\lbrace \lbrace v_1,v_2,v_5\rbrace ,\lbrace v_6\rbrace ,\lbrace v_3,v_4,v_7\rbrace ,\lbrace v_8\rbrace \rbrace $ is a $prc$-partition  for $C_8$, so $PRC(C_8)=4$.

	Next assume $n=11$. Considering Theorem \ref{bound-delta} and $\gamma_{p} (C_{11})=5$, it is not difficult to see that $PRC(C_{11}) \neq 6$. The partition $\lbrace \lbrace v_1,v_4,v_8,v_{11}\rbrace ,\lbrace v_2,v_3,v_6,v_{10}\rbrace ,\lbrace v_5\rbrace ,\lbrace v_9\rbrace, \lbrace v_7\rbrace \rbrace $ is a $prc$-partition  for $C_{11}$, so $PRC(C_{11})=5$.
	
	Finally, applying Lemma \ref{lem-c}, the desired result follows.  
\end{proof}
\section{Graphs with large perfect coalition number}
Our aim in this section is to characterize graphs with large perfect coalition number in some families of graphs. 
The following observation characterizes disconnected graphs with maximum perfect coalition number.
\begin{observation} \label{dis-n}
	Let $G$ be a disconnected graph of order $n$. Then  $PRC(G)=n$ if  and only if $G\simeq K_s \cup K_r$, for some positive integers  $r,s\geq 1$.
\end{observation}
\subsection{Graphs $G$ with $\delta(G)=1$ and $PRC(G)=\lvert V(G)\rvert$}
In this subsection, we characterize graphs with minimum degree one and  with maximum perfect coalition number. For this purpose,
we define the family of graphs $\mathcal{B}$  as follows.
\begin{definition}
	Let $G$ be a graph with $\delta (G)=1$. Let $x$ be a leaf of $G$ and let $y$ be the support vertex of $x$. Then $G \in \mathcal{B}$, if and only if $G$ has the following properties.
	\begin{enumerate}
		\item[(i)]  
		$N(y)$ is an independent set of $G$. 
		\item[(ii)]  
		$V(G) \setminus (N(y) \cup \lbrace x,y\rbrace)$ induces a clique in $G$. 
		\item[(iii)]  
		$[N(y),V(G) \setminus (N(y) \cup \lbrace x,y\rbrace)]$ is full.
	\end{enumerate}
\end{definition}
\begin{theorem} \label{thm-delta1}
	Let $G$ be a connected graph of order $n$ with $\delta (G)=1$. Let $x$ be a leaf of $G$, and let $y$ be the support vertex of $x$. Then $PRC(G)=n$ if and only if either $G=K_2$, or $G \in \mathcal{B}$. 
\end{theorem}
\begin{proof}
	Clearly, $PRC(K_2)=2$. Now  assume that $G \in \mathcal{B}$. Consider the singleton partition $\pi_1$ of $G$. Observe that for each vertex $v \in N(y)$, the set $\lbrace v\rbrace$ forms a $prc$-coalition with $\lbrace y\rbrace$, and for each vertex $v \in V(G) \setminus N(y) \cup \lbrace x,y\rbrace$, the set $\lbrace v\rbrace$ forms a $prc$-coalition with $\lbrace x\rbrace$. Hence, $\pi_1$ is a $prc$-partition of $G$, and so $PRC(G)=n$. Conversely, assume $PRC(G)=n$. Let $A=N(y)$, and $B=V(G) \setminus (A \cup \lbrace x,y\rbrace)$. If $A=\emptyset$ and $B=\emptyset$, then $G=K_2$, as desired. Now assume $G \neq K_2$. Since $G$ is connected, we have $A \neq \emptyset$. Note also that $B\neq \emptyset$, for otherwise, $G$ would be a star, and so $G$ would have no $PRC$-partition. Now  consider the $PRC$-partition $\pi_1$ of $G$.  Since $N(x)=\lbrace y\rbrace$, each set in $\pi_1 \setminus \lbrace \lbrace x\rbrace ,\lbrace y\rbrace\rbrace$ must form a $prc$-coalition with $\lbrace x\rbrace$ or $\lbrace y\rbrace$, to dominate $x$. 
	Now for each $v \in A$, the set $\lbrace v\rbrace$ cannot form a $prc$-coalition with $\lbrace x\rbrace$, since $y$ has more than one neighbors in $\lbrace v,x\rbrace$. Thus, $\lbrace v\rbrace$ forms a $prc$-coalition with $\lbrace y\rbrace$, implying that $[A,B]$ is full. Note also that $v$ has no neighbor in $A$, for otherwise, $\lbrace v,y\rbrace$ would not be a perfect dominating set. Therefore, $A$ is an independent set in $G$. Further, for each $u \in B$, the set $\lbrace u\rbrace$ must form a $prc$-coalition with $\lbrace x\rbrace$, since $\lbrace u,y\rbrace$ is not a perfect dominating set. Since $[\lbrace x\rbrace ,B]$ is empty, it follows that $B$ induces a clique in $G$. Hence, $G \in \mathcal{B}$.    
\end{proof}

\subsection{Triangle-free graphs $G$ with $PRC(G)=\lvert V(G)\rvert$}
In  this subsection, we characterize all triangle-free graphs  with maximum perfect coalition number. For this purpose,
we define the families of graphs $\mathcal{T}_1$ and $\mathcal{T}_2$  as follows.
\begin{definition} \label{family1}
	Let $\mathcal{T}_1$ represent the family of bipartite graphs $H- M$, where $H=K_{r,r}$ is a complete bipartite graph with $r \geq 2$, and $M$ is a perfect matching in $H$. 
\end{definition}
\begin{definition} \label{family2}
	Let $\mathcal{T}_2$ represent the family of bipartite graphs $H- M$, where $H=K_{r,s}$ is a complete bipartite graph with 
$r,s \geq 2$, and $M$ (possibly empty) is a  matching in $H$ such that $\lvert M\rvert < min \lbrace r,s\rbrace$.  
\end{definition}
Figure \ref{T1&T2} illustrates some examples of graphs in $\mathcal{T}_1 \cup \mathcal{T}_2$, where the vertices involved in $M$ are colored red.
 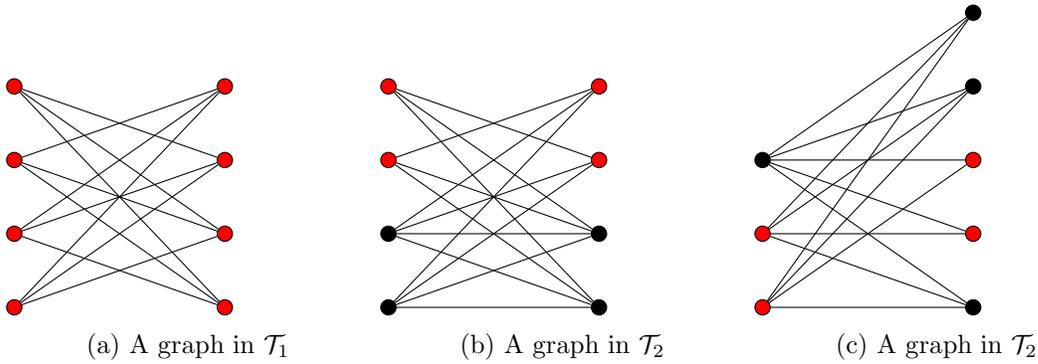
\begin{figure}[!htbp] 
	\centering
	\begin{subfigure}{0.3\textwidth}
		\begin{tikzpicture}[scale=0.14, transform shape]
			
			\node [draw, shape=circle,fill=red,scale=4] (v1) at  (0,0){};
			\node [draw, shape=circle,fill=red,scale=4] (v2) at  (0,7) {};
			\node [draw, shape=circle,fill=red,scale=4] (v3) at  (0,14) {};
			\node [draw, shape=circle,fill=red,scale=4] (v4) at  (0,21) {};
			\node [draw, shape=circle,fill=red,scale=4] (v5) at  (20,0) {};
			\node [draw, shape=circle,fill=red,scale=4] (v6) at  (20,7) {};
\node [draw, shape=circle,fill=red,scale=4] (v7) at  (20,14) {};
\node [draw, shape=circle,fill=red,scale=4] (v8) at  (20,21) {};			
			
			\draw(v1)--(v6);
			\draw(v1)--(v7);
			\draw(v1)--(v8);
			\draw(v2)--(v5);
			\draw(v2)--(v7);
					\draw(v2)--(v8);
			\draw(v3)--(v5);
			\draw(v3)--(v6);
			\draw(v3)--(v8);
			\draw(v4)--(v5);
					\draw(v4)--(v6);
			\draw(v4)--(v7);
		
		\end{tikzpicture}
		\caption{A graph in $\mathcal{T}_1$.}
		
	\end{subfigure}
	\begin{subfigure}{0.3\textwidth}
		\begin{tikzpicture}[scale=.14, transform shape]
			
			\node [draw, shape=circle,fill=black,scale=4] (v1) at  (0,0){};
			\node [draw, shape=circle,fill=black,scale=4] (v2) at  (0,7) {};
			\node [draw, shape=circle,fill=red,scale=4] (v3) at  (0,14) {};
			\node [draw, shape=circle,fill=red,scale=4] (v4) at  (0,21) {};
			\node [draw, shape=circle,fill=black,scale=4] (v5) at  (20,0) {};
			\node [draw, shape=circle,fill=black,scale=4] (v6) at  (20,7) {};
			\node [draw, shape=circle,fill=red,scale=4] (v7) at  (20,14) {};
			\node [draw, shape=circle,fill=red,scale=4] (v8) at  (20,21) {};			
			
			\draw(v1)--(v6);
			\draw(v1)--(v7);
			\draw(v1)--(v8);
			\draw(v2)--(v5);
			\draw(v2)--(v7);
			\draw(v2)--(v8);
			\draw(v3)--(v5);
			\draw(v3)--(v6);
			\draw(v3)--(v8);
			\draw(v4)--(v5);
			\draw(v4)--(v6);
			\draw(v4)--(v7);
			\draw(v1)--(v5);
\draw(v2)--(v6);			
		 				
		\end{tikzpicture}
		\caption{A graph in $\mathcal{T}_2$.}
	\end{subfigure}
	\begin{subfigure}{0.3\textwidth}
		\begin{tikzpicture}[scale=.14, transform shape]
			
			\node [draw, shape=circle,fill=red,scale=4] (v1) at  (0,0){};
			\node [draw, shape=circle,fill=red,scale=4] (v2) at  (0,7) {};
			\node [draw, shape=circle,fill=black,scale=4] (v3) at  (0,14) {};
	
			\node [draw, shape=circle,fill=black,scale=4] (v5) at  (20,0) {};
			\node [draw, shape=circle,fill=red,scale=4] (v6) at  (20,7) {};
			\node [draw, shape=circle,fill=red,scale=4] (v7) at  (20,14) {};
			\node [draw, shape=circle,fill=black,scale=4] (v8) at  (20,21) {};
						\node [draw, shape=circle,fill=black,scale=4] (v9) at  (20,28) {};						
			
			\draw(v1)--(v5);
			\draw(v1)--(v7);
			\draw(v1)--(v8);
			\draw(v1)--(v9);
			\draw(v2)--(v5);
			\draw(v2)--(v6);
			\draw(v2)--(v8);
			\draw(v2)--(v9);
			\draw(v3)--(v5);
			\draw(v3)--(v6);
			\draw(v3)--(v7);
			\draw(v3)--(v8);
			\draw(v3)--(v9);

		\end{tikzpicture}
		\caption{A graph in $\mathcal{T}_2$.}
	\end{subfigure}
	\caption{Some graphs in $\mathcal{T}_1 \cup \mathcal{T}_2$.}\label{T1&T2}
\end{figure}

\begin{theorem}
	Let $G$ be a triangle-free graph of order $n\geq 4$. Then $PRC(G)=n$ if  and only if $G \in \mathcal{T}_1 \cup \mathcal{T}_2$. 
\end{theorem}
\begin{proof}
	Suppose that $G \in \mathcal{T}_1 \cup \mathcal{T}_2$. Let $K_1 =\lbrace v_1,v_2,\dots ,v_{\lvert K_1\rvert} \rbrace$ and $K_2=\lbrace u_1,u_2,\dots ,u_{\lvert K_2\rvert} \rbrace$ denote the partite sets of $G$. First assume $G \in \mathcal{T}_1$. Assume, without loss of generality, that $u_i v_i \notin E(G)$, for each $1\leq i\leq \frac{n}{2}$. Now observe that the singleton partition $\pi_1$ of $G$ is a $prc$-partition for $G$, where the sets $\lbrace u_i\rbrace$ and $\lbrace v_i \rbrace$ are $prc$-partners, for each $1 \leq i \leq \frac{n}{2}$. Thus, $PRC(G)=n$. Now assume $G \in \mathcal{T}_2$. Consider the singleton partition $\pi_1$ of $G$. Further, assume that the vertices $v_1 \in K_1$ and $u_1 \in K_2$ are  not saturated by $M$. Now for each $v_i \in K_1$, if deg$_{G}(v_i)=\lvert K_2\rvert$, then the set $\lbrace v_i\rbrace$ forms a $prc$-coalition with $\lbrace u_1\rbrace$, otherwise,  $\lbrace v_i\rbrace$ forms a $prc$-coalition with $\lbrace u_j\rbrace$, where $v_i u_j \notin E(G)$. By symmetry, we can show that each singleton subset of $K_2$ has also a $prc$-partner in $\pi_1$.  Hence, $\pi_1$  is a $prc$-partition for $G$, and so $PRC(G)=n$.  
	Conversely, Assume that $PRC(G)=n$. We proceed with proving  the following claims.
	
	\begin{unnumbered}{Claim 1.}
		$g(G) \leq 6$.
	\end{unnumbered}
	\proof
	Suppose to the contrary that, $g(G)\geq 7$. Let  $C$ be a cycle of $G$ with order $g(G)$. Since $C$ has no perfect dominating set of order $2$, it follows that  for any vertex $v\in V(C)$, the set  $\lbrace v\rbrace$ is not a $prc$-partner of any set $\lbrace u\rbrace \subset V(C)$, so it must be a $prc$-partner of a set   $\lbrace u\rbrace \subseteq V(G) \setminus V(C)$. Thus, $\lbrace u\rbrace$  dominates $V(C) \setminus N_c [v]$, implying that $G$ contains triangles, a contradiction.
	\\
	
	An argument similar to the one presented above shows that $g(G) \neq 5$. 
	
	\begin{unnumbered}{Claim 2.}
		If $g(G)=6$, then $G=C_6$.
	\end{unnumbered}
	\proof
	Suppose to the contrary that, $G \neq C_6$. Let $C$ be a cycle of order $6$ in $G$, and let $v \in V(G)\setminus V(C)$. Note that the set $\lbrace v\rbrace$ cannot form a $prc$-partition with any set $\lbrace u\rbrace \subset V(C)$, for otherwise, $\lbrace v\rbrace$ would dominate $V(C) \setminus N[u]$, and so triangles would be created. Hence, $\lbrace v\rbrace$ must form a $prc$-coalition with a set $\lbrace u\rbrace \subseteq V(G) \setminus V(C)$. Now since each vertex in $V(C)$ is dominated by $\lbrace u,v\rbrace$, we can find a vertex in $\lbrace u,v\rbrace$ having at least three neighbors in $V(C)$, which creates cycles of order $3$ or $4$, a contradiction. Hence, $G=C_6$, and so $G \in \mathcal{T}_1$. 
	\\

	Note also that if $G$ has no cycle, then it follows from Observation \ref{dis-n} and Corollary \ref{tree-n} that $G \in \mathcal{T}_1 \cup \mathcal{T}_2$. Hence, for the rest of the proof, we may assume that $g(G)=4$. Consider the singleton $PRC(G)$-partition $\pi_1$. Let $C$ be a $4$-cycle in $G$ with $V(C)=\lbrace x,y,z,t\rbrace$ and $E(C)=\lbrace xy,yz,zt,tx\rbrace$. We distinguish two cases.
	
	\textbf{Case 1.} The set $\lbrace x\rbrace$ forms a $prc$-coalition with a singleton subset of $V(C)$. \\
	Observe that $\lbrace x\rbrace$ cannot form a $prc$-coalition with  $\lbrace z\rbrace$, so we may assume, 
by symmetry, that the sets $\lbrace x\rbrace$ and $\lbrace y\rbrace$ are $prc$-partners.
	Define $A=N(x) \setminus \lbrace t,y\rbrace$ and $B=N(y)\setminus \lbrace x,z\rbrace$. Observe that $A \cap B =\emptyset$. Now we consider the following subcases.
	
	\textbf{Subcase 1.1.} $A=\emptyset$ and $B=\emptyset$. It follows that $G=C_4$, and so $G \in \mathcal{T}_2$.
	
	\textbf{Subcase 1.2.} $A \neq \emptyset$ and $B=\emptyset$. If $[\lbrace z\rbrace ,A]$ is full, then $G \in \mathcal{T}_2$, as desired. Hence, we may assume that $[\lbrace z\rbrace ,A]$ is not full. Let $v \in A$ such that $vz \notin E(G)$. Note that the set $\lbrace v\rbrace$ has no $prc$-partner in $\lbrace \lbrace x\rbrace, \lbrace y\rbrace ,\lbrace t\rbrace \rbrace$. Now if $A=\lbrace v\rbrace$, then $G \in \mathcal{T}_2$, as desired. Hence, we may assume $A \neq \lbrace v\rbrace$. Note that $[\lbrace z\rbrace ,\lbrace A \setminus \lbrace v\rbrace]$ is full, for otherwise $\lbrace v\rbrace$ would have no $prc$-partner. Hence, $G \in \mathcal{T}_2$.
	
	\textbf{Subcase 1.3.} $A = \emptyset$ and $B \neq \emptyset$. Similar to the previous subcase, we can show that  $G \in \mathcal{T}_2$.
	
	\textbf{Subcase 1.4.} $A \neq \emptyset$ and $B \neq \emptyset$.
	Consider arbitrary vertices $v \in A$ and $u \in B$. Observe that $\lbrace v\rbrace$ has no $prc$-partner in $\lbrace \lbrace t\rbrace,\lbrace y\rbrace \rbrace$. Now it is not hard to verify that either $\lvert [\lbrace v\rbrace ,\lbrace x,z\rbrace \cup B] \rvert =\lvert B\rvert +2$, or $\lvert [\lbrace v\rbrace ,\lbrace x,z\rbrace \cup B] \rvert =\lvert B\rvert +1$. Similarly, we can show that either $\lvert [\lbrace u\rbrace ,\lbrace y,t\rbrace \cup A] \rvert =\lvert A\rvert +2$, or $\lvert [\lbrace u\rbrace ,\lbrace y,t\rbrace \cup A] \rvert =\lvert A\rvert +1$. Consequently, $G$ is a bipartite graph with partite sets $A \cup \lbrace y,t\rbrace$ and $B \cup \lbrace x,z\rbrace$ satisfying all properties of Definition \ref{family2}. Hence,  $G \in \mathcal{T}_2$.
	
	\textbf{Case 2.} The set $\lbrace x\rbrace$ does not form a $prc$-coalition with any singleton subset of $V(C)$. \\
	Let $\lbrace e\rbrace$ be a $prc$-partner of $\lbrace x\rbrace$. Define $A=N(x) \setminus \lbrace y,t,e\rbrace$ and $B=N(e) \setminus\lbrace z,x\rbrace$. Observe that $ez \in E(G)$ and $A \cap B=\emptyset$. We consider the following subcases.
	
	\textbf{Subcase 2.1.} $A=\emptyset$ and $B=\emptyset$. It is clear that $G \in \mathcal{T}_2$, as desired.
	
	\textbf{Subcase 2.2.} $A=\emptyset$ and $B \neq \emptyset$. Note that the set $\lbrace y\rbrace$ (or $\lbrace t\rbrace$) must form a $prc$-coalition with either a member of $\lbrace \lbrace x\rbrace ,\lbrace z\rbrace \rbrace$, or a singleton subset of $B$, implying that either $[\lbrace y\rbrace , B]$ is full, or $\lvert [\lbrace y\rbrace ,B]\rvert =\lvert B\rvert -1$. Now if $xe \in E(G)$, then for any subset $\lbrace u\rbrace \subseteq B$, the set $\lbrace u\rbrace$ must have a $prc$-partner in $\lbrace \lbrace y\rbrace ,\lbrace t\rbrace ,\lbrace e\rbrace \rbrace$, implying that either $[\lbrace u\rbrace ,\lbrace y,t,e\rbrace]$ is full, or $\lvert [\lbrace u\rbrace ,\lbrace y,t,e\rbrace] \rvert =2$. Otherwise, for any subset $\lbrace u\rbrace \subseteq B$, the set $\lbrace u\rbrace$ must have a $prc$-partner in $\lbrace \lbrace y\rbrace ,\lbrace t\rbrace  \rbrace$, implying again that either $[\lbrace u\rbrace ,\lbrace y,t,e\rbrace]$ is full, or $\lvert [\lbrace u\rbrace ,\lbrace y,t,e\rbrace] \rvert =2$. Hence, $G$ is  a bipartite graph with partite sets $\lbrace y,t,e\rbrace$ and $B \cup \lbrace x,z\rbrace$, satisfying all properties of Definition $\ref{family2}$, and so $G \in \mathcal{T}_2$.
	
	\textbf{Subcase 2.3.} $A \neq \emptyset$ and $B = \emptyset$. Note that for any subset $\lbrace v\rbrace \subseteq A$, the set $\lbrace v\rbrace$ must have a $prc$-partner in $\lbrace \lbrace x\rbrace ,\lbrace z\rbrace \rbrace$, implying that either $[\lbrace v\rbrace ,\lbrace x,z\rbrace]$ is full, or $\lvert [\lbrace v\rbrace ,\lbrace x,z\rbrace] \rvert =1$.  Now if the set $\lbrace z\rbrace$ has a $prc$-partner in $\lbrace \lbrace y\rbrace ,\lbrace t\rbrace ,\lbrace e\rbrace \rbrace$, then $[\lbrace z\rbrace , A]$ is full. Otherwise, $\lbrace z\rbrace$ must form a $prc$-coalition with a singleton subset in $A$, implying that either $[\lbrace z\rbrace ,A]$ is full, or $\lvert [\lbrace z\rbrace ,A]\rvert =\lvert A\rvert -1$. Note also that the set $\lbrace y\rbrace$ must have a $prc$-partner in $\lbrace \lbrace x\rbrace ,\lbrace z\rbrace \rbrace$. Thus, if $\lbrace z\rbrace$ has no $prc$-partner in $\lbrace \lbrace y\rbrace ,\lbrace t\rbrace ,\lbrace e\rbrace \rbrace$, then $\lbrace y\rbrace$ must form a $prc$-coalition with $\lbrace x\rbrace$, implying that $xe \in E(G)$, and so $[\lbrace x\rbrace ,\lbrace y,t,e\rbrace \cup A]$ is full. Consequently, $G$ is   a bipartite graph with partite sets $\lbrace x,z\rbrace$ and $A \cup \lbrace y,t,e\rbrace$, satisfying all properties of Definition $\ref{family2}$, and so $G \in \mathcal{T}_2$.
	
	\textbf{Subcase 2.4.} $A \neq \emptyset$ and $B \neq \emptyset$. Note that for any subset $\lbrace u\rbrace \subset B \cup \lbrace  z\rbrace$, the set $\lbrace u\rbrace$ must form a $prc$-partner with either a member of $\lbrace \lbrace y\rbrace ,\lbrace t\rbrace ,\lbrace e\rbrace \rbrace$, or a singleton subset of $A$, implying that either $[\lbrace u\rbrace ,A \cup \lbrace y,t,e\rbrace]$ is full, or $\lvert [\lbrace u\rbrace ,A \cup \lbrace y,t,e\rbrace]\rvert =\lvert A\rvert +2$. Further, for any subset $\lbrace v\rbrace \subset A \cup \lbrace  y,t\rbrace$, the set $\lbrace v\rbrace$ must form a $prc$-partner with either a member of $\lbrace \lbrace x\rbrace ,\lbrace z\rbrace  \rbrace$, or a singleton subset of $B$, implying that either $[\lbrace v\rbrace ,B \cup \lbrace x,z\rbrace]$ is full, or $\lvert [\lbrace v\rbrace ,B \cup \lbrace x,z\rbrace]\rvert =\lvert B\rvert +1$. Hence, $G$ is a complete bipartite graph with partite sets $K_1= A \cup \lbrace y,t,e\rbrace$ and $K_2 =B \cup \lbrace x,z\rbrace$ such that each vertex in $K_1$ has degree either $\lvert K_2\rvert$ or $\lvert K_2\rvert -1$, and each vertex in $K_2$ has degree either $\lvert K_1\rvert$ or $\lvert K_1\rvert -1$. To complete the proof, we show that if $K_1$ contains a vertex of degree $\lvert K_2\rvert$, then so does $K_2$, and vice versa. Assume first that $K_1$ contains a vertex $v$ with deg$(v)=\lvert K_2\rvert$.  If $v=e$, then the desired result follows. Hence, we may assume $v \neq e$, implying that $xe \notin E(G)$. Now  the set $\lbrace v\rbrace$ must form a $prc$-coalition with $\lbrace z\rbrace$ or a singleton subset $\lbrace u\rbrace \subseteq B$, implying, respectively, that $[\lbrace z\rbrace ,A]$ is full, or $[\lbrace u\rbrace ,A \cup \lbrace y,t\rbrace]$ is full, as desired. Conversely, assume $K_2$ contains a vertex $v$ with deg$(v)=\lvert K_1\rvert$. If $v=x$, then the desired result follows. Hence, we may assume $v \neq x$, implying that $xe \notin E(G)$. Now the set $\lbrace v\rbrace$ must form a $prc$-coalition with  a member of $\lbrace \lbrace y\rbrace ,\lbrace t\rbrace \rbrace$, say $\lbrace y\rbrace$, or a singleton subset $\lbrace u\rbrace \subseteq A$, implying, respectively, that $[\lbrace y\rbrace, B]$ is full, or $[\lbrace u\rbrace ,B \cup \lbrace z\rbrace]$ is full, as desired. Hence, $G \in \mathcal{T}_1 \cup \mathcal{T}_2$. This completes the proof.
\end{proof}
\subsection{Trees with large perfect coalition number}
In this subsection, we investigate trees with large perfect coalition number. As an immediate result from Theorem \ref{thm-delta1}, we have the following.  
\begin{corollary} \label{tree-n}
	Let $T$ be a tree of order $n$. Then $PRC(T)=n$ if and only if $T \in \lbrace P_1,P_2,P_4\rbrace$.
\end{corollary}
We define the tree $\mathcal{R}$ as shown below.
\begin{figure}[!htbp]
	\centering
	\begin{tikzpicture}[scale=.4, transform shape]
		\node [draw, shape=circle,fill=black] (v1) at  (0,0) {};
		\node [draw, shape=circle,fill=black] (v2) at  (0,2) {};
		\node [draw, shape=circle,fill=black] (v3) at  (-1,-2) {};

		\node [draw, shape=circle,fill=black] (v4) at  (1,-2) {};
		\node [draw, shape=circle,fill=black] (v5) at  (-2,-4) {};
		\node [draw, shape=circle,fill=black] (v6) at  (2,-4) {};
		
		\node [scale=3] at (0,2.8) {$v_1$};
		\node [scale=3] at (-0.8,0) {$v_2$};
		\node [scale=3] at (-1.8,-2) {$v_3$};
		\node [scale=3] at (1.8,-2) {$v_4$};
		\node [scale=3] at (-2.8,-4) {$v_5$};
		\node [scale=3] at (2.8,-4) {$v_6$};

		\draw (v1)--(v2);

		\draw (v1)--(v3);
		
		\draw(v1)--(v4);
		\draw(v3)--(v5);
		\draw(v4)--(v6);			
		
	\end{tikzpicture}
	\caption{The tree $\mathcal{R}$.}
\end{figure}
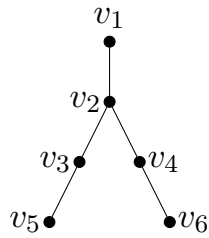
\begin{theorem}
	If $T \neq P_4$ is a tree of order $n>2$, then $PRC(T) \leq n-2$ with equality if and only if $T \in \lbrace P_5, P_6,P_7,\mathcal{R}\rbrace$. 
\end{theorem}
\begin{proof}
	By Corollary \ref{tree-n}, $PRC(T) \neq n$. Now we prove the following claim. 
		\begin{unnumbered}{Claim 1.}
	$PRC(T) \neq n-1$.	 
	\end{unnumbered}
	\proof
	
	 Suppose, to the contrary, that $PRC(T)=n-1$. Since stars are the only trees with a full vertex, and stars have no $prc$-partition, it follows that $T$ has no full vertex. Let $\pi$ be a $PRC(T)$-partition. Note that $\pi$ consists of a doubleton set, say $\lbrace u,v\rbrace$, and $n-2$ singleton sets. Let the set $\lbrace x\rbrace$ be a $prc$-partner of $\lbrace u,v\rbrace$. Further, let $A$ and $B$ denote the set of vertices in $V(T) \setminus \lbrace x,u,v\rbrace$ that are dominated by $\lbrace x\rbrace$ and $\lbrace u,v\rbrace$, respectively. Since $\lbrace x,u,v\rbrace$ is a perfect dominating set of $T$, we have $A \cap B =\emptyset$. Also, since $\lbrace u,v\rbrace$ is not a dominating set, it follows that $A \neq \emptyset$. Now we show  that $B \neq \emptyset$. Suppose that the converse is true. Since $T$ is connected, $x$ has a neighbor in $\lbrace u,v\rbrace$. If $uv \in E(T)$, then it is easy to check that for any set $\lbrace a\rbrace \subseteq A$, $\lbrace a\rbrace$ has no $prc$-partner. Hence, we may assume that $uv \notin E(T)$. Now it follows from connectedness of $T$ that $[\lbrace x\rbrace ,\lbrace u,v\rbrace]$ is full, implying that $\lbrace x\rbrace$ is a full vertex, a contradiction. Hence, $B \neq \emptyset$. Let $\lbrace a\rbrace$ be an arbitrary singleton subset of $A$. Now we consider the following cases.
	
	\textbf{Case 1.}  $uv \in E(T)$ and $[\lbrace x\rbrace ,\lbrace u,v\rbrace] \neq \emptyset$.
	
	We may assume, by symmetry, that $xv \in E(T)$.  Since $u$ is not dominated by $\lbrace a\rbrace \cup \lbrace x\rbrace$, the sets $\lbrace a\rbrace$ and $\lbrace x\rbrace$ are not $prc$-partners. Further, since $x$ has two neighbors in $\lbrace a\rbrace \cup \lbrace u,v\rbrace$, the sets $\lbrace a\rbrace$ and $\lbrace u,v\rbrace$ are not $prc$-partners. Note also that for any singleton subset $\lbrace b\rbrace \subseteq B$, $\lbrace a\rbrace$ and $\lbrace b\rbrace$ are not $prc$-partners, for otherwise, a cycle or cycles would be created. Hence, $\lbrace a\rbrace$ has no $prc$-partner, a contradiction.
	
	\textbf{Case 2.}  $uv \notin E(T)$ and $[\lbrace x\rbrace ,\lbrace u,v\rbrace] \neq \emptyset$.    We divide this case into two subcases.
	
	\textbf{Subcase 2.1.} $\lvert [\lbrace x\rbrace ,\lbrace u,v\rbrace] \rvert=1$.
	
	We may assume, by symmetry, that $xv \in E(T)$. 
	Since $u$ is not dominated by $\lbrace a\rbrace \cup \lbrace x\rbrace$, the sets $\lbrace a\rbrace$ and $\lbrace x\rbrace$ are not $prc$-partners. Further, since $x$ has more than one neighbors in $\lbrace a\rbrace \cup \lbrace u,v\rbrace$, the sets $\lbrace a\rbrace$ and $\lbrace u,v\rbrace$ are not $prc$-partners. Note also that no singleton subset of $B$ dominates $\lbrace u,v\rbrace$ implying that no singleton subset of $B$ can form a perfect coalition with $\lbrace a\rbrace$. Hence, $\lbrace a\rbrace$ has no $prc$-partner, a contradiction.
	
	\textbf{Subcase 2.2.}  $\lvert [\lbrace x\rbrace ,\lbrace u,v\rbrace] \rvert=2$.
	
	Since $x$ has more than one neighbors in $\lbrace a\rbrace \cup \lbrace u,v\rbrace$, the sets $\lbrace a\rbrace$ and $\lbrace u,v\rbrace$ are not $prc$-partners. Further, no singleton subset of $B$ dominates $\lbrace u,v\rbrace$ implying that no singleton subset of $B$ can form a perfect coalition with $\lbrace a\rbrace$. Not also that the sets $\lbrace a\rbrace$ and $\lbrace x\rbrace$ are not $prc$-partners, for otherwise, a cycle or cycles would be created.  Hence, $\lbrace a\rbrace$ has no $prc$-partner, a contradiction.
	
	\textbf{Case 3.} $[\lbrace x\rbrace ,\lbrace u,v\rbrace]$ is empty.
	
	Since $T$ is connected, $[A,B]\neq \emptyset$. Let $ab \in E(T)$, where $b \in B$.  Since $u$ is not dominated by $\lbrace a\rbrace \cup \lbrace x\rbrace$, the sets $\lbrace a\rbrace$ and $\lbrace x\rbrace$ are not $prc$-partners. Further, since $b$ has more than one neighbors in $\lbrace a\rbrace \cup \lbrace u,v\rbrace$, the sets $\lbrace a\rbrace$ and $\lbrace u,v\rbrace$ are not $prc$-partners. Note also that no singleton subset of $B$ dominates $\lbrace u,v\rbrace$ implying that no singleton subset of $B$ can form a perfect coalition with $\lbrace a\rbrace$. Hence, $\lbrace a\rbrace$ has no $prc$-partner, a contradiction. Hence, $PRC(T) \neq n-1$, and so $PRC(T) \leq n-2$. \\
	
	Now we prove the second part of the theorem. By Theorem \ref{prc-path}, $PRC(P_5)=3$, $PRC(P_6)=4$ and $PRC(P_7)=5$. Also, the partition $\lbrace \lbrace v_1\rbrace ,\lbrace v_2\rbrace ,\lbrace v_3,v_4\rbrace ,\lbrace v_5,v_6\rbrace \rbrace $ is a $prc$-partition for $\mathcal{R}$, so $PRC(\mathcal{R})=4$.
	Conversely, assume that $PRC(T)=n-2$. Let $\pi$ be a $PRC(T)$-partition. Consider two cases. 
	
	\textbf{Case 1.} $\pi$ consists of a tripleton set  $\lbrace u,v,t\rbrace$, and $n-3$ singleton sets.
	
	We show that this case is impossible. Suppose that the converse is true. Let the set $\lbrace x\rbrace$ be a $prc$-partner of $\lbrace u,v,t\rbrace$. Further, let $A$ and $B$ denote the set of vertices in $V(T) \setminus \lbrace x,u,v,t\rbrace$ that are dominated by $\lbrace x\rbrace$ and $\lbrace u,v,t\rbrace$, respectively. Since $\lbrace x,u,v,t\rbrace$ is a perfect dominating set of $T$, we have $A \cap B =\emptyset$. Since $\lbrace u,v,t\rbrace$ is not a dominating set, it follows that $A \neq \emptyset$. Now we show that $B \neq \emptyset$. Suppose that the converse is true. Since $T$ is connected, $x$ has a neighbor in $\lbrace u,v,t\rbrace$, implying that for any set $\lbrace a\rbrace \subseteq A$, $\lbrace a\rbrace$ cannot form a $prc$-coalition with $\lbrace u,v,t\rbrace$. Thus, $\lbrace a\rbrace$ must form a $prc$-coalition with $\lbrace x\rbrace$, implying that $[\lbrace x\rbrace ,\lbrace u,v,t\rbrace]$ is full. Thus, $x$ is a full vertex, which is a contradiction. Hence, $B \neq \emptyset$. Let $\lbrace a\rbrace$ and $\lbrace b\rbrace$ be arbitrary singleton subsets of $A$ and $B$, respectively. Since $b$ has exactly one neighbor in $\lbrace u,v,t\rbrace$, and $a$ has no neighbor in $\lbrace u,v,t\rbrace$, it follows that $\lbrace a\rbrace$ cannot be a $prc$-partner of $\lbrace b\rbrace$. Now we show that the sets $\lbrace a\rbrace$ and $\lbrace x\rbrace$ are not $prc$-partners. Suppose that the converse is true. It follows that $[\lbrace x\rbrace ,\lbrace u,v,t\rbrace]$ is full, and that $ab \in E(T)$. By symmetry, we may assume that $bu \in E(T)$. Now observe that the set $\lbrace x,u,b,a\rbrace$ induces a cycle, a contradiction. It remains to show that  $\lbrace a\rbrace$ cannot form a $prc$-partition with $\lbrace u,v,t\rbrace$. If $x$ has a neighbor in $\lbrace u,v,t\rbrace$, then $x$ has more than one neighbors in $\lbrace a\rbrace \cup \lbrace u,v,t\rbrace$, implying that the sets $\lbrace a\rbrace$ and $\lbrace u,v,t\rbrace$ are not $prc$-partners. Hence, we may assume that $[\lbrace x\rbrace ,\lbrace u,v,t\rbrace] =\emptyset$. Since $T$ is connected, it follows that $a$ has a neighbor in $B$. Let $ab \in E(T)$. Now we observe that $b$ has more that one neighbors in $\lbrace a\rbrace \cup \lbrace u,v,t\rbrace$. Hence, $\lbrace a\rbrace$ cannot form a $prc$-partition with $\lbrace u,v,t\rbrace$, and so this case is impossible.
	
	\textbf{Case 2.} $\pi$ consists of two doubleton sets, say $\lbrace u,v\rbrace$ and $\lbrace t,z\rbrace$, and $n-4$ singleton sets.
	
	We divide this case into two subcases.
	
	\textbf{Subcase 2.1.} $\lbrace u,v\rbrace$ and $\lbrace t,z\rbrace$ are $prc$-partners.
	
	Let $A$ and $B$ denote the set of vertices in $V(T) \setminus \lbrace u,v,t,z\rbrace$ that are dominated by $\lbrace u,v\rbrace$ and $\lbrace t,z\rbrace$, respectively. Since $\lbrace u,v,t,z\rbrace$ is a perfect dominating set of $T$, we have $A \cap B =\emptyset$. Observe that the sets $A$ and $B$ cannot both be empty sets.
	
	\begin{unnumbered}{Claim 2.}
		If $A =\emptyset$, then $T=P_5$. 
	\end{unnumbered}
	\proof
	Since $\lbrace t,z\rbrace$ is not a dominating set, there exists a vertex in $\lbrace u,v\rbrace$, say $u$, that has no neighbor in $\lbrace t,z\rbrace$. Now since $T$ is connected, it follows that $uv \in E(T)$, and that $v$ has a neighbor in $\lbrace t,z\rbrace$. Assume, by symmetry that, $vt \in E(T).$ Let $\lbrace b\rbrace$ be an arbitrary singleton subset of $B$. Since $\lbrace b\rbrace $ is not a $prc$-partner of $\lbrace t,z\rbrace$, it must form a $prc$-coalition with $\lbrace u,v\rbrace$, implying that $tb \notin E(T)$, and that $[\lbrace z\rbrace ,\lbrace u,v\rbrace]=\emptyset$.  Choosing $b$ arbitrarily, we deduce that $t$ has no neighbor in $B$, implying that $[\lbrace z\rbrace ,B]$ is full. Now, for $T$ to be connected, we must have $tz \in E(T)$. Note also that $B=\lbrace b\rbrace$, for otherwise $\lbrace b\rbrace$ would have no $prc$-partner. Hence, $T=P_5$.
	
	By symmetry, if $B =\emptyset$, then $T=P_5$. 
	For the rest of this subsection, we may assume that $A \neq \emptyset$ and $B \neq \emptyset$. 
	Let $\lbrace a\rbrace$ and $\lbrace b\rbrace$ be arbitrary singleton subsets of $A$ and $B$, respectively. Note that $[\lbrace u,v\rbrace ,\lbrace t,z\rbrace] \neq \emptyset$, for otherwise, it is easy to verify that the set $\lbrace a\rbrace$ would not have a $prc$-partner.\\
	
	\begin{unnumbered}{Claim 3.}
		$[\lbrace t\rbrace ,\lbrace u,v\rbrace]$ is not full.
	\end{unnumbered}
	\proof
	Suppose that the converse is true. Assume, by symmetry, $au \in E(T)$. Observe that no singleton subset of $A$ can form a $prc$-coalition with $\lbrace a\rbrace$. Since $u$ has more than one neighbors in $\lbrace a\rbrace \cup \lbrace t,z\rbrace$, the sets $\lbrace a\rbrace$ and $\lbrace t,z\rbrace$ are not $prc$-partners. Further, since $t$ has more than one neighbors in $\lbrace u,v\rbrace$, the sets $\lbrace a\rbrace$ and $\lbrace u,v\rbrace$ are not $prc$-partners, and since $\lbrace a\rbrace$ has exactly one neighbor in $\lbrace u,v\rbrace$, no singleton subset of $B$ can form a $prc$-coalition with $\lbrace a\rbrace$. Hence, $\lbrace a\rbrace$ has no $prc$-partner, a contradiction.
	\\
	
	Therefore, by symmetry, none of the sets $[\lbrace z\rbrace ,\lbrace u,v\rbrace]$, $[\lbrace u\rbrace ,\lbrace t,z\rbrace]$ and $[\lbrace v\rbrace ,\lbrace t,z\rbrace]$ are full.
	\begin{unnumbered}{Claim 4.}
		If $\lvert [\lbrace u,v\rbrace ,\lbrace t,z\rbrace] \rvert = 1$, then $T=P_6$. 
	\end{unnumbered}
	\proof

	Let by symmetry that, $[\lbrace u,v\rbrace ,\lbrace t,z\rbrace]= \lbrace vt\rbrace$. 
Observe that $A=\lbrace a\rbrace$ and $B=\lbrace b\rbrace$.
So, it is not hard to verify that the set $\lbrace a\rbrace$ can only form a $prc$-partition with $\lbrace t,z\rbrace$, so $[\lbrace a\rbrace ,\lbrace u,v\rbrace]=\lbrace au\rbrace$. By symmetry, $[\lbrace b\rbrace ,\lbrace t,z\rbrace]= \lbrace bz\rbrace$.  Further, note that $ab \notin E(T)$, for otherwise, $b$ would have more than one neighbors in $\lbrace a\rbrace \cup \lbrace t,z\rbrace$, contradicting the fact that $\lbrace a\rbrace$ and $\lbrace t,z\rbrace$ are $prc$-partners. Thus, for $T$ to be connected, we must have $\lbrace uv,tz\rbrace  \subset E(T)$, and so $T=P_6$. \\
	
	Finally, we may assume by symmetry that,  $ [\lbrace u,v\rbrace ,\lbrace t,z\rbrace] = \lbrace vt,uz\rbrace$. Assume by symmetry that, $av \in E(T)$. Since $v$ has more that one neighbors in 
$\lbrace t,z\rbrace \cup \lbrace a\rbrace$, the sets $\lbrace t,z\rbrace$ and $\lbrace a\rbrace$ are not 
$prc$-partners. Further, Observe that $\lbrace a\rbrace$ cannot form a $prc$-coalition with  
any singleton subset of $B$. Hence, $\lbrace a\rbrace$ must form a $prc$-coalition with $\lbrace u,v\rbrace$, implying that  $ab \in E(T)$. Thus, $bt \notin E(T)$, and so $bz \in E(T)$. Since $\lvert A\rvert =\lvert B\rvert =1$,  $T=P_6$. \\
	
	\textbf{Subcase 2.2.} $\lbrace u,v\rbrace$ and $\lbrace t,z\rbrace$ are not $prc$-partners.		
	
	Let $\lbrace w\rbrace$ be a $prc$-partner of $\lbrace u,v\rbrace$. 	
	Let $A$ and $B$ denote the set of vertices in $V(T) \setminus \lbrace u,v,w\rbrace$ that are dominated by $\lbrace w\rbrace$ and $\lbrace u,v\rbrace$, respectively. Since $\lbrace u,v,w\rbrace$ is a perfect dominating set of $T$, we have $A \cap B =\emptyset$. Also, since the set  $\lbrace u,v\rbrace$  is not  dominating sets, it follows that $A \neq \emptyset$. Now we show that $B \neq \emptyset$. Suppose, to the contrary, that $B=\emptyset$. Let $\lbrace a\rbrace \subseteq A$. Since $\lbrace w\rbrace$ is not a dominating set, $[\lbrace w\rbrace ,\lbrace u,v\rbrace]$ is not full. Now since $T$ is connected, it follows that $uv \in E(T)$ and that $w$ has a neighbor in $\lbrace u,v\rbrace$. Assume, by symmetry, $vw \in E(T)$. Now observe that the set $\lbrace a\rbrace$ has no $prc$-partner, a contradiction.

	\begin{unnumbered}{Claim 5.}
		$\lbrace t,z\rbrace \nsubseteq A$.
	\end{unnumbered}
	\proof
	Suppose, to the contrary, that  $\lbrace t,z\rbrace \subseteq A$. Observe that $\lbrace t,z\rbrace$ cannot form a $prc$-partition with any singleton subset of $A$. Also by our assumption, the sets $\lbrace u,v\rbrace$ and $\lbrace t,z\rbrace$ are not $prc$-partners. Further, for any arbitrary singleton subset  $\lbrace b\rbrace \subseteq B$, since $b$ has one neighbor in $\lbrace u,v\rbrace$, it follows that the sets $\lbrace b\rbrace$ and $\lbrace t,z\rbrace$ are not $prc$-partners.  Note also that $\lbrace t,z\rbrace$ does not form a $prc$-partition with $\lbrace w\rbrace$, for otherwise, a cycle or cycles would be created. Hence,  the set $\lbrace t,z\rbrace$ has no $prc$-partner, a contradiction. \\ 
	
	\begin{unnumbered}{Claim 6.}
		$\lvert A\rvert \leq 2$.
	\end{unnumbered}
	\proof
	Suppose, to the contrary, that $\lvert A\rvert \geq 3$. Let $\lbrace a,c,d\rbrace \subseteq A$ such that $a \notin \lbrace t,z\rbrace$. Assume first that $[\lbrace w\rbrace ,\lbrace u,v\rbrace]$ is full. Observe that the set $\lbrace a\rbrace$ cannot form a $prc$-coalition with any singleton subset of $A$ or $B$. Further, since $w$ has more than one neighbors in $\lbrace a\rbrace \cup \lbrace u,v\rbrace$, it follows that the sets $\lbrace a\rbrace$ and $\lbrace u,v\rbrace$ are not $prc$-partners. Note also that $\lbrace a\rbrace$ does not form a $prc$-partition with any member of $\lbrace \lbrace w\rbrace, \lbrace t,z\rbrace \rbrace$, for otherwise, it is easily seen that a cycle or cycles would be created. Thus, $\lbrace a\rbrace$ has no $prc$-partner, a contradiction. Hence, we may assume that $[\lbrace w\rbrace ,\lbrace u,v\rbrace]$ is not full. It is easy to verify that $\lbrace a\rbrace$ cannot form a $prc$-coalition with  any member of $\lbrace \lbrace w\rbrace ,\lbrace u,v\rbrace\rbrace$, or any singleton subset of  $A \cup  (B \setminus \lbrace t,z\rbrace)$.  Hence, $\lbrace a\rbrace$ must form a $prc$-coalition with $\lbrace t,z\rbrace$. Note that if $\lbrace t,z\rbrace \cap A \neq \emptyset$, then $\lbrace a\rbrace$ cannot form a $prc$-coalition with $\lbrace t,z\rbrace$, so $\lbrace t,z\rbrace \cap A =\emptyset$. By symmetry, each vertex in $\lbrace c,d\rbrace$ must form a $prc$-coalition with $\lbrace t,z\rbrace$ as well. Thus, each vertex in $\lbrace a,c,d\rbrace$ has a neighbor in $\lbrace t,z\rbrace$, so there exists a vertex in $\lbrace t,z\rbrace$ having more that one neighbors in $\lbrace a,c,d\rbrace$, which creates cycle, a contradiction. \\
	
	\begin{unnumbered}{Claim 7.}
		If $\lvert A\rvert =2$, then $B=\lbrace t,z\rbrace$.
	\end{unnumbered}
	\proof
	First we show that $\lbrace t,z\rbrace \cap A =\emptyset$. Suppose, to the contrary, that $\lbrace t,z\rbrace \cap A  \neq \emptyset$. Let by symmetry that $t \in A$, where $A =\lbrace t,a\rbrace$. Now it is not hard verifying that the set $\lbrace a\rbrace$ has no $prc$-partner, a contradiction. 
	Now we show that $B=\lbrace t,z\rbrace$. Suppose, to the contrary, that $B \neq \lbrace t,z\rbrace$. Let $b \in B \setminus \lbrace t,z\rbrace$ and $A=\lbrace a,c\rbrace$. Assume first that $[\lbrace w\rbrace ,\lbrace u,v\rbrace]$ is full. It is easy to check that the set $\lbrace b\rbrace$ cannot form a $prc$-coalition with any member of $\lbrace \lbrace u,v\rbrace ,\lbrace t,z\rbrace \rbrace$, or  any singleton subset of $A \cup (B \setminus \lbrace t,z\rbrace)$. Further, $\lbrace b\rbrace$ does not form a $prc$-coalition with $\lbrace w\rbrace$, for otherwise, $[\lbrace b\rbrace ,\lbrace t,z\rbrace]$ would be full, and so we would have cycle. Thus, $\lbrace b\rbrace$ has no $prc$-partner, a contradiction. Hence, we may assume that $[\lbrace w\rbrace ,\lbrace u,v\rbrace]$ is not full. It is easy to verify that no singleton subset of $A$  has a $prc$-partner in $\lbrace \lbrace w\rbrace ,\lbrace c\rbrace ,\lbrace u,v\rbrace \rbrace$, implying that each of them must form a $prc$-coalition with $\lbrace t,z\rbrace$. Thus, both $a$ and $b$ have a neighbor in $\lbrace t,z\rbrace$, which contradicts the fact that $T$ is a tree.\\
	
	Now assume that $\lvert A\rvert =2$, and so $B=\lbrace t,z\rbrace$.
	Let $A=\lbrace a,b\rbrace$. Observe that the set $\lbrace a\rbrace$ (or $\lbrace b\rbrace$) cannot form a $prc$-coalition with any member of $\lbrace \lbrace w\rbrace ,\lbrace u,v
	\rbrace \rbrace$, so it must form a $prc$-coalition with $\lbrace t,z\rbrace$. So, in order to avoid cycles, we must have $T[\lbrace a,b,t,z\rbrace] =K_2 \cup K_2$, $T[\lbrace u,v,t,z\rbrace]=K_2 \cup K_2$ and $E(T[\lbrace u,v,w\rbrace])=\emptyset$. Hence, $T \simeq P_7$.
	
	\begin{unnumbered}{Claim 8.}
	If $\lvert A\rvert =1$, then $\rvert B\rvert \leq 3$.
\end{unnumbered}
\proof

	Suppose, to the contrary, that $\lvert B\rvert \geq 4$. Let $\lbrace c,d\rbrace \subseteq B \setminus \lbrace
   t,z\rbrace$. We proceed with the following claims.
  
   	\begin{unnumbered}{Claim 8.1.}
   	$[\lbrace w\rbrace ,\lbrace u,v\rbrace]$ is not empty.			
   \end{unnumbered}
   \proof
   Suppose, to the contrary, that  $[\lbrace w\rbrace ,\lbrace u,v\rbrace] =\emptyset$. If $\lbrace t,z\rbrace \cap A =\emptyset$, then it is easy to check that the set $\lbrace c\rbrace$ has no $prc$-partner, a contradiction. Hence, we may assume $\lbrace t,z\rbrace \cap A  \neq \emptyset$. Let by symmetry that $A=\lbrace t\rbrace$. One can verify that the set $\lbrace c\rbrace$ cannot form a $prc$-coalition with any member of $\lbrace \lbrace u,v\rbrace ,\lbrace w\rbrace \rbrace$ or any singleton subset of $B$, implying that $\lbrace c\rbrace$ must form a $prc$-coalition with $\lbrace t,z\rbrace$. Let $B=\lbrace c,d,e,z\rbrace$ and let by symmetry that $zv \in E(T)$. Now the cycle $T[\lbrace t,e,u,d\rbrace]$ would be created, a contradiction. \\

  	\begin{unnumbered}{Claim 8.2.}
	$[\lbrace w\rbrace ,\lbrace u,v\rbrace]$ is not full.			
\end{unnumbered}
\proof
 Suppose that the converse is true. If  $\lbrace t,z\rbrace \cap A =\emptyset$, then it is not hard verifying that the set $\lbrace c\rbrace$ has no $prc$-partner,  a contradiction. Hence, we may assume, by symmetry, that $A=\lbrace t\rbrace$. Observe that the set $\lbrace c\rbrace$ cannot form a $prc$-coalition with $\lbrace u,v\rbrace$ or any singleton subset of $B$. Note also that $\lbrace c\rbrace$ has no $prc$-partner in $\lbrace \lbrace w\rbrace ,\lbrace t,z\rbrace \rbrace$, for otherwise, a cycle or cycles would be created. Hence, $\lbrace c\rbrace$ has no $prc$-partner, a contradiction. \\
 
       Hence, we may assume $\lvert [\lbrace w\rbrace ,\lbrace u,v\rbrace ]\rvert =1$. Let by symmetry that $vw \in E(T)$. Assume first that $\lbrace t,z\rbrace \cap A =\emptyset$. Let $A=\lbrace a\rbrace$. It is easily seen that the set $\lbrace a\rbrace$ cannot for a $prc$-partition with any member of $\lbrace \lbrace w\rbrace ,\lbrace u,v\rbrace \rbrace$ or any singleton subset of $B$. Thus, it must form a $prc$-coalition with $\lbrace t,z\rbrace$, implying that $T[\lbrace u,v,t,z\rbrace] = K_2 \cup K_2$. Observe that the set $\lbrace c\rbrace$ cannot form a $prc$-coalition with any member of $\lbrace A , \lbrace t,z\rbrace \rbrace$ or any singleton subset of $B$. Note also that the sets $\lbrace c\rbrace$ and $\lbrace w\rbrace$ are not $prc$-partners, for otherwise, $[\lbrace c\rbrace ,\lbrace t,z,u\rbrace]$ would be full, and so we would have a cycle, a contradiction. Thus, $\lbrace c\rbrace$ must form a $prc$-coalition with $\lbrace u,v\rbrace$. By symmetry, $\lbrace d\rbrace$ must form a $prc$-coalition with $\lbrace u,v\rbrace$ as well. Therefore,  $[\lbrace a\rbrace ,\lbrace c,d\rbrace]$ is full, creating a cycle or cycles, a contradiction. Hence we may assume, by symmetry, that $A=\lbrace t\rbrace$. Note that the set $\lbrace t,z\rbrace$ cannot form a $prc$-coalition with any singleton subset of $B$, for otherwise, a cycle or cycles would be created. Further, observe that $\lbrace t,z\rbrace$ does not form a $prc$-coalition with $\lbrace w\rbrace$. Hence, $\lbrace t,z\rbrace$ has no $prc$-partner, a contradiction. This completes the proof of Claim 8. To complete the proof of  theorem, we consider the following claims. 
 	\begin{unnumbered}{Claim 9.}
	If $\lvert A\rvert =1$ and $\lvert B\rvert =2$, then $T \in \lbrace P_6, \mathcal{R} \rbrace$.			
\end{unnumbered}
\proof

	First assume  $A \cap \lbrace t,z\rbrace =\emptyset$. Let $A=\lbrace a\rbrace$. It is easy to verify that $\lbrace a\rbrace$ cannot form a $prc$-coalition with any member of $\lbrace \lbrace w\rbrace ,\lbrace u,v\rbrace\rbrace$, so it must form a $prc$-coalition with $\lbrace t,z\rbrace$, implying that each vertex in $\lbrace t,z\rbrace$ has a distinct neighbor in $\lbrace u,v\rbrace$. If $[\lbrace w\rbrace ,\lbrace u,v\rbrace]$ is full, then  $T \simeq \mathcal{R}$. Otherwise, suppose first that  $[\lbrace w\rbrace ,\lbrace u,v\rbrace]$ is empty.  Since $T$ is connected, $a$ has a neighbor in $\lbrace t,z\rbrace$. If $[\lbrace a\rbrace,\lbrace t,z\rbrace]$ is full, then $T \simeq \mathcal{R}$. Otherwise, either $uv \in E(T)$ or $tz \in E(T)$, implying that either $T \simeq \mathcal{R}$ or $T \simeq P_6$.  Now suppose, by symmetry, that $[\lbrace w\rbrace ,\lbrace u,v\rbrace]=\lbrace vw\rbrace$ and $zv\in E(T)$. If $a$ has no neighbor in $\lbrace t,z\rbrace$, then either $uv \in E(T)$ or $tz \in E(T)$, implying that either $T \simeq \mathcal{R}$ or $T \simeq P_6$. Otherwise, we have $at \in E(T)$, and so $T \simeq P_6$.
	Next assume $A \cap \lbrace t,z\rbrace  \neq \emptyset$. Let by symmetry that $A=\lbrace t\rbrace$ and $B=\lbrace b,z\rbrace$. Suppose first that $[\lbrace w\rbrace ,\lbrace u,v\rbrace] =\emptyset$. Then the set $\lbrace b\rbrace$ must form a $prc$-coalition with $\lbrace t,z\rbrace$, implying that each vertex in $\lbrace b,z\rbrace$ has a distinct neighbor in $\lbrace u,v\rbrace$. Since $T$ is connected, $t$ has a neighbor in $\lbrace b,z\rbrace$. Let by symmetry that $tz \in E(T)$. Now for $T$ to be connected, we must have either $bz \in E(T)$ or $uv \in E(T)$, implying respectively that either $T =\mathcal{R}$ or $T=P_6$. Next observe  that if  $[\lbrace w\rbrace ,\lbrace u,v\rbrace]$ is full, then $T=\mathcal{R}$. Finally, suppose that $\lvert [\lbrace w\rbrace ,\lbrace u,v\rbrace]\rvert =1$. Let by symmetry that $wv \in E(T)$. If the set $\lbrace t,z\rbrace$ forms a $prc$-coalition with $\lbrace w \rbrace$, then $zu \in E(T)$. In this case, either $tz \in E(T)$ or $bz \in E(T)$, implying that $T=P_6$. Otherwise, $\lbrace t,z\rbrace$ must form a $prc$-coalition with $\lbrace b\rbrace$. In this case, it is easy to verify that either $T =\mathcal{R}$ or $T=P_6$. \\
	
 	\begin{unnumbered}{Claim 10.}
	If $\lvert A\rvert =1$ and $\lvert B\rvert =3$, then $T=  P_7$.			
\end{unnumbered}
\proof

	First  assume $A \cap \lbrace t,z\rbrace =\emptyset$. Let $A=\lbrace a\rbrace$ and $B=\lbrace t,z, b\rbrace$. If $[\lbrace w\rbrace ,\lbrace u,v\rbrace]=\emptyset$, then it is easy to check that the set $\lbrace b\rbrace$ has no $prc$-partner, a contradiction. Else if  $[\lbrace w\rbrace ,\lbrace u,v\rbrace]$ is full, then considering the fact that $T$ has no cycles, we observe that the set $\lbrace a\rbrace$ has no $prc$-partner, a contradiction. Hence, we may assume, by symmetry, that $[\lbrace w\rbrace ,\lbrace u,v\rbrace]=\lbrace vw\rbrace$. Note that the set $\lbrace a\rbrace$ can only form a $prc$-coalition with $\lbrace t,z\rbrace$, implying that each vertex in $\lbrace t,z\rbrace$ has a distinct neighbor in $\lbrace u,v\rbrace$ and that $b$ is dominated by $\lbrace a\rbrace \cup \lbrace t,z\rbrace$. If $ab \notin E(T)$, then it is not hard verifying that the set $\lbrace b\rbrace$ would have no $prc$-partner, so $ab \in E(T)$, implying that $ub \in E(T)$, and so $T \simeq P_7$.  Now assume $A \cap \lbrace t,z\rbrace  \neq \emptyset$. Let by symmetry that $A=\lbrace t\rbrace$, $B=\lbrace z,b,c\rbrace$ and $vz \in E(T)$. We show that $[\lbrace w\rbrace ,\lbrace u,v\rbrace] \neq \emptyset$. Suppose the opposite is right. Observe that each of the sets $\lbrace b\rbrace$ and $\lbrace c\rbrace$ cannot  form a $prc$-coalition with any member of $\lbrace \lbrace w\rbrace ,\lbrace
	 u,v\rbrace \rbrace$, implying that they must form a $prc$-coalition with $\lbrace t,z\rbrace$, which is again impossible since a cycle would be created. Note that $[\lbrace w\rbrace ,\lbrace u,v\rbrace]$ is not full, for otherwise, it is easy to check that the set $\lbrace t,z\rbrace$ would have no $prc$-partner. Hence, we may assume, by symmetry $[\lbrace w\rbrace ,\lbrace u,v\rbrace]= \lbrace vw\rbrace$. Note that no member of $\lbrace \lbrace t,z\rbrace ,\lbrace b\rbrace ,\lbrace c\rbrace \rbrace$ can form a $prc$-coalition with $\lbrace w\rbrace$, for otherwise, we would have more that six edges, contradicting the fact that $T$ is a tree. Thus, each of the sets $\lbrace b\rbrace$ and $\lbrace c\rbrace$ must form a $prc$-coalition with $\lbrace t,z\rbrace$, which is again impossible since we would have more than six edges. This completes the proof.
\end{proof}

\section{Discussion and conclusions}
In this paper, we introduced the notion of perfect coalitions in graph, which extends the notion of coalitions in graphs. We obtained an upper bound on the number of coalitions involving any member in a $prc$-partition of a graph. This helped us determine the perfect coalition number of paths and cycles. We also obtained some results regarding graphs with large perfect coalition number. We end the paper with the following open problems.
\begin{p}
	Characterize graphs admitting a $prc$-partition.
\end{p}
\begin{p}
	Characterize graphs $G$ in which the equality $PRC(G)=C(G)$ holds.
\end{p}
\begin{p}
	Characterize graphs $G$ with $\delta (G)=2$ and $PRC(G)=\lvert V(G)\rvert$. 
\end{p}
	
\begin{flushleft}
	\textbf{{\large Conflicts of interest}}\vspace{-3.5mm}
\end{flushleft}
The authors declare that they have no conflict of interest.

\section*{Acknowledgments}
A part of this research was conducted by the first author while he was a visiting professor in the Department of Combinatorics and Optimization at the University of Waterloo in Summer 2024.

\end{document}